\documentclass[11pt,a4paper]{amsart}

\usepackage{amsmath}
\usepackage{amssymb}
\usepackage{url}

\numberwithin{equation}{section}

\newcommand{\R}{{\mathbb R}}

\DeclareMathOperator{\stab}{stab}

\newtheorem*{theorem*}{Theorem}
\newtheorem*{corollary*}{Corollary}
\newtheorem{theorem}{Theorem}

\newtheorem{lemma}[theorem]{Lemma}
\newtheorem{conjecture}[theorem]{Conjecture}
\theoremstyle{definition}

\begin{document}
\title{Face-subgroups of permutation polytopes}
\date{February 27, 2015}

\author{Christian Haase}
\email{haase@math.fu-berlin.de}
\address{Mathematisches Institut, Freie Universit\"at Berlin,
  14195 Berlin, Germany}



\vspace*{-20mm}
\maketitle
\vspace*{-6mm}
\thispagestyle{empty}
\section{The Conjecture}
A permutation polytope $P(G)$ is the convex hull, in $\R^{n \times
  n}$, of a subgroup $G$ of the group $S_n$ of $n \times n$
permutation  matrices. This is a convex geometric invariant of a
permutation representation, and it yields various numerical invariants
like dimension, volume, diameter, $f$-vector, etc.
Permutation polytopes and their faces were studied in \cite{BHNP07}.

Let $[n] := \{1, \ldots, n\} = \bigsqcup I_k$ be a partition of the
ground set. We define the stabilizer of this partition as
\[
\stab(G;(I_k)_k) := \{ \sigma \in G : \sigma(I_k)=I_k
\text{ for all } k \} \leq G\,.
\]

\begin{lemma} \label{lemma:stab-is-face}
  $P(\stab(G;(I_k)_k))$ is a face of $P(G)$.
\end{lemma}

In~\cite{BHNP07} we formulated the following conjecture.

\begin{conjecture}
  Let $G \le S_n$. Suppose $H \le G$ is a subgroup such that $P(H)
  \preceq P(G)$ is a face. Then there is a partition $[n] = \bigsqcup
  I_k$ so that $H = \stab(G;(I_k)_k)$.
  \label{conj}
\end{conjecture}

As evidence for the conjecture, we observed that it holds for $G =
S_n$ as well as for cyclic $G$. Later, Jessica Nowack and Daniel
Heinrich verified the conjecture for the dihedral groups $D_n$,
compare \cite[Prop.~6.1]{dihedralPermutationPolytopes}.

Subsequent efforts to better understand the inequality description of
permutation polytopes, even of abelian groups, can lead one to believe
that there is hardly anything one could say about general permutation
polytopes~\cite{cyclicPermutationPolytopesFPSAC}.

\section{The Proof}
For $i \in [n]$ we denote the orbit under $G$ by $G . i$, and the
stabilizer of $i$ by $G_i$.
\begin{lemma} \label{lemma:barycenter}
  Let $A := \frac{1}{|G|} \sum_{g \in G} g$ be the vertex-barycenter
  of $G \subset \R^{n \times n}$.
  Then $A$ has entries
  \[
  a_{ij} = \begin{cases}
    \frac{1}{|G . i|} & \text{ if } G . i = G . j \\
    0 & \text{ else.}
  \end{cases}
  \]
  In particular, $A$ only depends on the orbit structure of $G$.
\end{lemma}
\begin{proof}
  The $ij$-entry of $\sum_{g \in G} g$ counts $g \in G$ with
  $g(j)=i$. This number is either zero or equal to $|G_i|$. Now use
  $|G|=|G_i|\cdot|G.i|$.
\end{proof}
\begin{theorem}
  Conjecture~\ref{conj} is true.
\end{theorem}
\begin{proof}
  Let $H \le G$ be a subgroup such that $P(H) \preceq P(G)$ is a
  face. Let $[n] = \bigsqcup I_k$ be the partition of $[n]$ into
  $H$-orbits, and set $\hat H := \stab(G;(I_k)_k)$. Then $H$ and $\hat
  H$ have the same orbit partition. Thus, $P(H)$ and $P(\hat H)$ are
  faces of $P(G)$ by assumption and by Lemma~\ref{lemma:stab-is-face},
  respectively. They share a relative interior point by
  Lemma~\ref{lemma:barycenter}. This implies $P(H)=P(\hat H)$.
\end{proof}

\bibliographystyle{amsplain}
\bibliography{alles,haasi}

\end{document}